\documentclass[12pt,a4paper,reqno]{article}
\usepackage[left=1cm,right=1cm,top=2cm,bottom=2cm]{geometry}
\usepackage[utf8]{inputenc}
\usepackage[english]{babel}
\usepackage{amsmath}
\usepackage{amsfonts}
\usepackage{amssymb}
\usepackage{color}
\usepackage{mathtools}
\usepackage{bbm}
\usepackage{bm}
\usepackage{enumerate}
\usepackage{mathrsfs}
\usepackage{esint}
\usepackage{dsfont}
\usepackage{amsthm}
\usepackage{cite}
\usepackage{hyperref}
\hypersetup{%
        colorlinks=true,
        citecolor={black},
        linkcolor={black},
        urlcolor={black},
        pdfview=FitBH}

\renewcommand*{\AA}{\mathcal{A}}

\newcommand*{\NN}{\mathbb{N}}
\newcommand*{\CC}{\mathbb{C}}
\newcommand*{\HH}{\mathcal{H}}
\newcommand*{\DD}{\mathcal{D}}

\newcommand*{\RR}{\mathbb{R}}
\renewcommand*{\SS}{\mathcal{S}}

\newcommand*{\KK}{\mathcal{K}}

\newcommand*{\XX}{\mathcal{X}}

\renewcommand*{\i}{\mathrm{i}}
\newcommand*{\0}{\theta}

\newcommand*{\smin}{\setminus}
\newcommand*{\deq}{:=}
\newcommand*{\eps}{\varepsilon}

\renewcommand*{\d}{\mathrm{d}}

\newcommand*{\pd}{\partial}

\newcommand*{\vphi}{\varphi}

\renewcommand*{\a}{\alpha}
\renewcommand*{\b}{\beta}
\renewcommand*{\$}{\mathfrak{s}}
\renewcommand*{\ln}{\log}
\renewcommand*{\square}{\diamond}

\newcommand\blfootnote[1]{%
  \begingroup
  \renewcommand\thefootnote{}\footnote{#1}%
  \addtocounter{footnote}{-1}%
  \endgroup
}

\newcommand*{\kw}[1]{\emph{#1}}

\DeclareMathOperator{\tr}{tr}

\DeclareMathOperator{\capac}{cap}

\theoremstyle{plain}
\newtheorem{theorem}{Theorem}[section]
\newtheorem{proposition}[theorem]{Proposition}
\newtheorem{lemma}[theorem]{Lemma}

\theoremstyle{definition}
\newtheorem{D}[theorem]{Definition}

\theoremstyle{remark}
\newtheorem{remark}[theorem]{Remark}
\newtheorem{example}[theorem]{Example}

\title{Small volume expansion of the splitting of\\multiple Neumann Laplacian eigenvalues\\due to a grounded inclusion in two dimensions}

\author{Alexander Dabrowski\footnote{Department of Mathematics, ETH Zurich.
The author gratefully acknowledges Prof.~H.~Ammari for the fruitful conversations, Dr.~S.~Yu for the implementation of the multipole expansion method, and Prof.~G.~S.~Alberti for the code which started the numerical experimentation.
}}

\begin{document}

\maketitle


\abstract{The first terms of the small volume asymptotic expansion for the splitting of Neumann boundary condition Laplacian eigenvalues due to a grounded inclusion of size $\eps$ are derived.
An explicit formula to compute the first term from the eigenvalues and eigenfunctions of the unperturbed domain, the inclusion size and position is given.
As a consequence, when an eigenvalue of double multiplicity splits in two distinct eigenvalues, one decays like $O(1/\log(\eps))$, the other like $O(\eps^2)$.}

\section{Introduction}
\label{sect:Intro}

Consider a planar domain $\Omega$ and let $\omega^2$ be an eigenvalue of the negative Laplacian on $\Omega$ with Neumann boundary condition.
Suppose a small inclusion $D = z + \eps B$ (where $z \in \Omega$, $|B| = |\Omega|$, and $\eps$ is small) is inserted inside $\Omega$.
This may cause the eigenvalue $\omega_\eps^2$ of the perturbed domain $\Omega \smin D$, with Neumann condition on $\pd \Omega$ and Dirichlet on $\pd D$, to vary in value or in multiplicity with respect to the original eigenvalue $\omega^2$.
Asymptotic formulae of the eigenvalue perturbation with respect to the size of the inclusion have been derived in the `80s in \cite{Ozawa}, \cite{Besson}.
In particular it has been shown that if $\omega^2$ is simple and $u $ is the associated $L^2$-normalized eigenfunction, it holds
\begin{equation}
\label{eq:oldomegaeps-omega0}
\omega_\eps^2 - \omega^2 = -\dfrac{2 \pi |u(z)|^2}{\log \eps} + o(1/\log(\eps)) .
\end{equation}
More recently, Gohberg-Sigal theory for meromorphic operators applied to the integral equation formulation of the eigenvalue problem has led to new results (see \cite{AmmariKangLimZribi}, \cite{AmmariTriki04}).
In this paper we use these results to improve \eqref{eq:oldomegaeps-omega0}, by calculating explicitly the terms up to $o(\eps)$ and generalizing it to the case of multiple eigenvalues.
As a consequence of our derivation, we have that for perturbed eigenvalues $\omega_{\eps,1}^2  < \omega_{\eps,2}^2 $, splitted from a double eigenvalue $\omega^2$ of the original domain $\Omega$, it holds
\begin{align*}
\omega_{\eps,2} - \omega = &~  -\dfrac{C_1}{\ln (\eps) + C_2} + O(\eps^2),\\
\omega_{\eps,1} - \omega = &~  O(\eps^2),
\end{align*}
where $C_1$ and $C_2$ do not depend on $\eps$ and can be explicitly calculated from $z$ and the eigenvalues, eigenfunctions of $\Omega$.
Similar formulae for eigenvalues of higher multiplicity can be derived.

More in detail the structure of the paper is as follows. 
After introducing in section \ref{sect:Setting} the precise setting of the problem and notation, in section \ref{sect:IntFormulation} we recall the equivalent formulation of the Laplacian eigenvalues as characteristic values of an appropriate integral operator.
An asymptotic expansion of this integral operators can be obtained by expanding in Taylor series the free space fundamental solution.
Gohberg-Sigal theory then provides a link between the eigenvalue splitting and the traces of these integral operators through power sum polynomials with roots in the eigenvalues splitting.

In the core section \ref{sect:Computations}, explicit terms for the small volume expansion of these power sum polynomials are derived by using properties of layer potentials.
The key step here is the filtering of the spectral decomposition of the Neumann function using the residue theorem to obtain geometric-like series which can be summed.
A tentative proposal for formal automated computation of higher order coefficients is given in section \ref{sect:Higherorderproposal}.

Finally in section \ref{sect:Results} some interesting consequences for special cases and a brief validation with numerical experiments are provided.

\subsection{Main tools and notation}

\label{sect:Setting}
\paragraph{The eigenvalue problem}

Let $\Omega$ be a bounded domain in $\RR^2$ with connected and piecewise smooth boundary.
It is well known that the eigenvalues of the negative Laplacian on $\Omega$ with Neumann boundary condition are non-negative, have finite multiplicity and can be arranged in an increasing divergent sequence
$$0 = \omega_0^2 < \omega_1^2 < \omega_2^2 < \dots < \omega_k^2 \to \infty .$$ 
For each index $i$, let $m_i$ be the multiplicity of $\omega_i$.
We choose the associated eigenfunctions $u_{i,1} \dots u_{i,m_i}$ to be orthonormal in $L^2$.
We thus have
$$\begin{cases}
(\Delta + \omega_i^2) u_{i,j} = 0 & \text{ in } \Omega,\\
\dfrac{\pd u_{i,j}}{\pd \nu} = 0 & \text{ on } \pd \Omega,
\end{cases}$$
and $$\int_\Omega u_{i,j} u_{k,l} = \begin{cases}
1 & \text{ if } i=k \text{ and } j=l,\\
0 & \text{ otherwise}.
\end{cases}$$
We will occasionally use the notation
$$U_i := (u_{i,1} \dots u_{i,m_i}).$$

\paragraph{Free space fundamental solution}

The free space fundamental solution for Helmholtz equation $(\Delta  + \omega^2) u = 0$ is a function $\Gamma_\omega$ s.t. for any $x,y \in \RR^2$, it holds
$$
(\Delta_x + \omega^2) \Gamma_\omega(x,y) = \delta_y(x) ,
$$
where $\delta_y$ is the Dirac delta function at $y$.
We adopt as fundamental solution
$$\Gamma_\omega(x,y) \deq
\begin{cases}
\dfrac{1}{2 \pi} \ln |x-y| & \text{ if } \omega = 0,\\[12pt]
\dfrac{1}{4} Y_0(\omega |x-y|) & \text{ otherwise},
\end{cases} $$
where $Y_0$ is the Bessel function of the second kind and order $0$; it can be defined by the power series
$$Y_0(t) \deq \dfrac{2}{\pi} \sum_{n=0}^{\infty} (-1)^n \left(\dfrac{t^{n}}{2^n n!}\right)^2  \ln (\eta_n t),$$ 
with $\ln \eta_n \deq \text{\textit{Euler-Mascheroni constant}} - \ln 2 + \sum_{k=1}^n \frac{1}{k}.$

\paragraph{Layer potentials}
Given $\phi \in  L^2(\pd \Omega)$, we define the operators 
\begin{align*}
\text{(Single layer potential)} \qquad \SS_\Omega^\omega[\phi](x) \deq &~ \int_{\pd \Omega} \Gamma_\omega(x,y) \phi(y) \ \d \sigma(y) \qquad \text{ for } x \in \RR^2, \\
\text{(Double layer potential)} \qquad \DD_\Omega^\omega[\phi](x) \deq &~ \int_{\pd \Omega} \dfrac{\pd \Gamma_\omega(x,y)}{\pd \nu(y)} \phi(y) \ \d \sigma(y) \qquad \text{ for } x \in \RR^2 \smin \pd \Omega, \\
\text{(Neumann-Poincar\'e operator)} \qquad \KK_\Omega^\omega[\phi](x) \deq &~ \int_{\pd \Omega} \dfrac{\pd \Gamma_\omega(x,y)}{\pd \nu(y)}  \phi(y) \ \d \sigma(y) \qquad \text{ for } x \in \pd \Omega.
\end{align*}
For their properties and extensive applications in the theory of boundary value problems we refer to \cite{ColtonKress, Verchota}.

\paragraph{Capacity of a set}
The single layer potential can be used to define the capacity of a set as follows
(see also \cite{ArmitageGardiner}).
It can be shown that there exists a unique couple $(\vphi_{\capac}, a) \in L^2(\pd \Omega) \times \RR$ which solves
$$\begin{cases}
\SS_\Omega^0[\vphi_{\capac}](x) \equiv a & \forall x \in \pd \Omega,\\
\int_{\pd \Omega} \vphi_{\capac} = 1.
\end{cases}$$
The logarithmic capacity of $\pd \Omega$ is then defined as $\capac \pd \Omega \deq e^{2 \pi a}.$

\begin{remark}
For $\Omega = B_1$ a unit disk, writing $\theta$ as the angle in the usual parametrization of $\pd \Omega$, after a lengthy calculation we have
$$\SS_{B_1}^0 [e^{\i n \theta}](t) =
\begin{cases}
0 & \text{ if } n = 0,\\
- \dfrac{1}{n} e^{\i n t} & \text{otherwise}.
\end{cases} $$
Thus we have an explicit expression of $\SS_{B_1}^0$ in the Fourier basis of $L^2(\pd B_1)$.
Notice however that the fact that $\SS_{B_1}^0[1]=0$ causes the non-invertibilty of $ \SS_{B_1}$.
However, if we consider $\phi \mapsto \SS_\Omega^0[\phi] + \lambda \int_{\pd \Omega} \phi$, we see that this operator is always invertible from $L^2(\pd \Omega)$ to $H^1(\pd \Omega)$.
This is still true for a domain $\Omega$ as in our assumptions (see \cite[Theorem 4.11]{Verchota} for more details).
\end{remark}

\paragraph{Fundamental solution for a bounded domain}
The Neumann function $N_\Omega^\omega$ is defined as the solution of
$$\begin{cases}
(\Delta_x + \omega^2) N_\Omega^\omega(x,z) = \delta_z(x) & \text{ for } x \in \Omega, \\
\dfrac{\pd N_\Omega^\omega(x,z)}{\pd \nu(x)} = 0 &  \text{ for } x \in \pd \Omega,
\end{cases}$$
where $\omega \in \CC$ is not one of the eigenvalues $\omega_i$, and $z \in \Omega$.
It has the spectral representation
$$N_\Omega^\omega(x,z) = \sum_{j=1}^\infty \dfrac{U_j(x) \cdot U_j(z)}{\omega^2 - \omega_j^2}, $$
where the convergence of the series to $N_\Omega^\omega$ in general is only in $L^2$ (see \cite[\textit{expansion theorems}]{CourantHilbert}).
By integrating $N_\Omega^\omega$ against test functions in $L^2(\pd \Omega)$ and using properties of layer potentials one can show that
$$\left( I/2 - \KK_\Omega^\omega \right) ^{-1} [\Gamma_\omega(\cdot,z)](x) = N_\Omega^\omega(x,z).$$
The Neumann function has a logarithmic singularity, in particular
\begin{equation}
\label{eq:NeumannSpectDecomp}
N_\Omega^\omega(x,z) = \dfrac{1}{2 \pi} \ln |x-z| + R_{\Omega}^\omega(x,z) \quad \forall x \neq z,
\end{equation}
with $R_{\Omega}^\omega$ continuous on $\Omega \times \Omega.$
(For more details on the last two results, see \cite[section 2.3.5]{thebook}.)

\paragraph{The perturbed eigenvalue problem}

Let $B$ be a bounded domain with piecewise smooth boundary, with area $|B| = |\Omega|$, and centered at the origin in the sense that
$$\int_{\pd B} y_1 \, \d \sigma(y_1,y_2) = \int_{\pd B} y_2 \, \d \sigma(y_1,y_2) = 0.$$
We fix for the rest of the paper a point $z \in \Omega$, a scaling factor $0<\eps \ll 1$ and an index $\0 \in \NN.$
Suppose then that the domain $\Omega $ is perturbed by inserting a grounded inclusion $D \deq z + \eps B$ inside $\Omega$.
This causes the eigenvalue $\omega_\0^2$ to split into $m_\0$ (possibly distinct) eigenvalues $\omega_{\eps,1}^2 \leq \dots \leq \omega_{\eps, m_\0}^2$ with associated eigenfunctions  $u_{\eps,1} \dots u_{\eps, m_\0}$.
This means that for $j=1 \dots m_\0$,
$$\begin{cases}
(\Delta + \omega_{\eps,j}^2) u_{\eps,j} = 0 & \text{ in } \Omega \smin D, \\
u_{\eps,j} = 0 & \text{ on } \pd D, \\
\dfrac{\pd u_{\eps,j}}{\pd \nu} = 0 & \text{ on } \pd \Omega.
\end{cases}$$
It has been shown in \cite{RauchTaylor} that under our assumptions $\omega_{\eps,j}^2 \to \omega_\0^2$ as $\eps \to 0$.
To find an asymptotic expansion of $\omega_{\eps,j}^2 - \omega_\0^2$ in terms of $\eps$ for $j=1 \dots m_\0$, we will transform this eigenvalue problem into an equivalent integral equation formulation.

\paragraph{Nonstandard notation}
For clarity, we adopt the symbol $\square$ to indicate the function variable of an operator evaluated at a point; e.g. $ \DD_\Omega^\omega[\square](z) $ indicates a map which takes a function in $L^2(\pd \Omega)$ and returns a number in $\RR$.

We indicate as $\displaystyle \oint $ the normalized complex path integral $\dfrac{1}{2 \pi \i} \displaystyle \int $.

\subsection{Integral formulation}
\label{sect:IntFormulation}

Define $\AA_\eps(\omega)$ as
$$ \CC \ni \omega \quad \mapsto \quad \AA_\eps(\omega) \deq \begin{pmatrix}
I/2 - \KK_\Omega^\omega  & -\SS_B^\omega \\
\DD_\Omega^\omega & \SS_B^\omega
\end{pmatrix},
\vspace{6pt}$$
meaning that for any fixed $\omega \in \CC$, 
$\AA_\eps(\omega)$ is the operator which takes $\phi \in L^2(\pd \Omega), \psi \in L^2(\pd B) $ to
$$\begin{pmatrix}
(I/2 - \KK_\Omega^\omega)[\phi](x) -\SS_B^\omega[\psi](x) \\[10pt]
\DD_\Omega^\omega[\phi](x) + \SS_B^\omega [\psi](x)
\end{pmatrix} \quad \in \quad \begin{matrix}
 L^2(\pd \Omega) \\ \times \\ \, L^2(\pd B).
 \end{matrix} $$

By expanding the fundamental solution in Taylor series in $\eps$, one can show that $\AA_\eps = \sum_{n=0}^\infty  \eps^n \HH_n$ (i.e. the series converges in operator norm), where
\begin{align*}
\HH_0 \deq &~ \begin{pmatrix}
I/2 - \KK_\Omega^\omega  & -\Gamma_\omega(x,z) \int_{\pd B} \square (y) \ \d \sigma(y) \\
\DD_\Omega^\omega[\square](z) & S
\end{pmatrix}, \\
\HH_{n} \deq &~ \begin{pmatrix}
0 &  (-1)^{n+1} \sum_{|\a|=n} (\pd^\a \Gamma_\omega)(x,z) \int_{\pd B} y^\a \square (y) \ \d \sigma(y) \\
\sum_{|\a|=n} (\pd^\a \DD_\Omega^\omega[\square])(z) x^\a  & \XX_n
\end{pmatrix}, 
\end{align*}
with
\begin{align}
\nonumber
\XX_n \deq &~ \begin{cases}
\dfrac{\omega^n}{2^{n+1} n!\pi} \displaystyle \int_{\pd B} \ln(\eta_{\frac{n}{2}} \omega \eps |x-y|) |x-y|^n \square (y) \ \d \sigma(y) & n \text{ even}, \\
0 & n \text{ odd},
\end{cases} \\[6pt]
\label{eq:DefS}
S \deq &~ \dfrac{1}{2 \pi} \displaystyle \int_{\pd B} \ln (\eta_0 \omega \eps |x-y| ) \square(y) \ \d \sigma(y).
\end{align}

A study of the properties of $A_\eps$ can be found in \cite[chapter 1 and section 3.1]{thebook}).
In the next proposition we collect the properties which will be used in the following discussion.
Recall that $\omega \in \RR$ is a characteristic value of $A_\eps$ if the null-space of $A_\eps (\omega)$ contains some non-zero function. 

\begin{proposition}
The following results hold:
\begin{enumerate}
\item $\omega \mapsto \AA_\eps(\omega) $ is analytic on $\CC \smin \i \RR^-$ and $\omega \mapsto \AA_\eps(\omega)^{-1} $ is meromorphic in $ \CC$,
\item $\omega_\0$ is a characteristic value of $ \HH_0$ and a simple pole of $\AA_\eps^{-1}$, 
\item $(\omega_{\eps,j})_{j=1}^{m_\0}$ are among the characteristic values of $\AA_\eps$,
\item There is an open neighbourhood $V$ (which we fix for the rest of the paper) of $\omega_\0$ s.t. $\omega_{\eps,j} \in V$ for $j=1 \dots m_\0$, and no other characteristic values of $\AA_\eps$ are in $V$.
\end{enumerate}
\end{proposition}

Consider now the power sum polynomials

$$p_l \deq \sum_{j=1}^{m_\0} (\omega_{\eps,j} - \omega_\0)^l.$$ 
By properties of symmetric polynomials we can express $\omega_{\eps,1} - \omega_\0 \dots \omega_{\eps,m_\0} - \omega_\0$ as roots of a polynomial $z^{m_\0} + c_{1} z^{m_\0 - 1} + \dots + c_{m_\0}$, where the coefficients $c_k$ are themselves polynomials in $p_j$;
in particular they can be recovered from the recurrence relation
$$p_{l+m_\0} + c_1 p_{l+m_\0-1} + \dots + c_{m_\0} p_{l} \qquad \text{ for } l=0 \dots m_\0-1. $$

\begin{example}
\label{ex:splittingp1p2}
If $m_\0 =1$ we have $$\omega_{\eps,1} - \omega_\0 = p_1,$$
while if  $m_\0 =2$ then
\begin{align*}
\omega_{\eps,2} - \omega_\0 = \dfrac{p_1 + \sqrt{2p_2 - p_1^2}}{2}, \qquad
\omega_{\eps,1} - \omega_\0 = \dfrac{p_1 - \sqrt{2p_2 - p_1^2}}{2}.
\end{align*}
\end{example}
Thus we have reduced the problem of finding an asymptotic expansion $\omega_{\eps,j}^2 - \omega_\0^2$  to finding an asymptotic expansion for $p_l$.
Before computing $p_l $ we recall some crucial concepts from Gohberg-Sigal theory.

Recall that if $A$ is a finite range operator on an infinite dimensional space, its trace $\tr A$ is defined as the trace of $A$ restricted to the finite dimensional space where $A$ is non zero.

\begin{proposition}
The following results hold:
\begin{enumerate}
\item 
Suppose $A_1, A_2, A_3, A_4$ are finite dimensional operators. Then $$\tr \begin{pmatrix}
A_1 & A_2\\
A_3 & A_4
\end{pmatrix} = \tr A_1 + \tr A_4.$$
\item Suppose $B, C$ are operator valued maps defined on $U $, a neighborhood of a common singularity $\omega_0 \in \CC.$
If $B,C$ are analytic in $U \smin \omega_0$ and have only finite dimensional operators in the negative terms of their Laurent expansion in $\omega_0$, then $\int_{\pd U } B(\omega)C(\omega) \, \d \omega $ is finite dimensional and 
$$\tr \oint_{\pd U } B(\omega) C(\omega) \, \d \omega =  \tr \oint_{\pd U } C(\omega) B(\omega) \, \d \omega .$$
\item If $P_\omega$ is a projection on a one dimensional subspace of $L^2$ generated by a function $f_\omega$, then
$$\tr \oint_{\pd V} f_\omega P_\omega \, \d \omega = \oint_{\pd V} P_\omega f_\omega \, \d \omega .$$
\end{enumerate}
\end{proposition}

An application of the argument principle for operator valued maps (for its formulation see \cite{GohbergSigal}) leads to the following crucial representation.

\begin{theorem}
\label{T:AsympExpPLGen}
The asymptotic expansion of $p_l$ in $\eps$ can be expanded as
\begin{align*}
p_l =  &~ \tr \oint_{\pd V} (\omega - \omega_\0)^l \HH_0(\omega)^{-1} \pd_\omega \HH_0(\omega) \, \d \omega \\
&~ + l \sum_{n=1}^{\infty} \eps^n \sum_{j=1}^n \dfrac{(-1)^j}{j}  \tr  \oint_{\pd V}  (\omega - \omega_\0)^{l-1} (\HH_0(\omega)^{-1})^n \Big( \sum_{k_1 + \dots + k_j = n} \HH_{k_1}(\omega) \dots \HH_{k_n}(\omega) \Big) \, \d \omega.
\end{align*}
\end{theorem}
The previous expression can be obtained by following the same steps in the proof of \cite[Theorem 3.9]{thebook}.

\section{Computations for explicit formulae}
\label{sect:Computations}

We first isolate the quantities playing a key role in the expansion of $p_l$ in the following constants.
\begin{D}
\label{def:genCapacitytr}
Let $\a, \b$ be multi-indices in $\NN^2$.
The \kw{generalized capacity} of $B$ of order $(\a, \b)$ is
$$s_{\a,\b} \deq  (-1)^{|\a|+|\b|+1} \int_{\pd B} y^\b S^{-1}[x^\a](y) \, \d \sigma(y). $$
We also introduce
\begin{align}
\label{eq:definet0r0}
t_\omega \deq &~ \dfrac{U_\0(z) \cdot \DD_\Omega^{\omega}[U_\0](z)}{\omega + \omega_\0}, \qquad
r_\omega \deq  \sum_{\substack{j=1 \\ j \neq \0}}^{\infty} \dfrac{U_j(z)  \cdot \DD_\Omega^{\omega}[U_j](z)}{\omega^2 - \omega_j^2}.
\end{align}
\end{D} 
In the subsequent discussion we will often  indicate the  generalized capacity of order $(\a, 0)$ as $s_\a$ instead of $s_{\a,0}$.

\begin{remark}
We collect some useful properties of the quantities introduced in the previous defintion:
\begin{itemize}
\item The generalized capacity of order zero can be rewritten explicitly in terms of $\eps$ and the capacity as
$$s_0 = -\int_{\pd B} S^{-1}[1] = -\left( \dfrac{\ln(\eta_0 \omega_\0 \eps)}{2 \pi} + \ln \capac \pd B \right)^{-1}. $$
\item It holds $s_{\a,\b} = 0$ if $|\a| + |\b|$ is odd.
This is a consequence of the fact that $\vphi$ is even/odd if and only if $\SS_B^0(\vphi)$ is even/odd (as functions parametrized on $\pd B)$.
\item By exploiting the spectral expansion of the Neumann function \eqref{eq:NeumannSpectDecomp}, we have that
$$ \DD_\Omega^\omega[N_\Omega^\omega(\cdot, z)](z) = \dfrac{t_\omega}{\omega - \omega_\0} + r_\omega,$$
where we were able to exchange series and integral since $\pd \Gamma_\omega(z, \cdot) / \pd \nu(\cdot) \in L^2(\pd \Omega)$.
In the following proof, this identity will enable us to rewrite the expansion in \eqref{T:AsympExpPLGen} in terms of $t_{\omega_\0}$ and $r_{\omega_\0}$.
\end{itemize}
\end{remark}

\subsection{Zero order term}

\begin{lemma}
\label{ZeroOrderTerm}
The zero order term in the expansion in $\eps$ of $p_l$ is
\begin{equation}
\label{eq:AsympExppl}
\left(\dfrac{t_{\omega_\0} }{ 1/s_0 - r_{\omega_\0}} \right)^l .
\end{equation}
\end{lemma}

\begin{proof}
By Theorem \ref{T:AsympExpPLGen}, our problem reduces to compute explicitly
\begin{equation*}
\tr \oint_{\pd V} (\omega - \omega_\0)^l \HH_0(\omega)^{-1} \pd_\omega \HH_0(\omega) \, \d \omega.
\end{equation*}
To make further computations clearer and more concise, we rename
\begin{equation}
\label{proof:1stTermRenames}
\begin{matrix}
 A \deq   I/2 - \KK_\Omega^\omega[\square](x), \quad &  \Gamma \deq   \Gamma_\omega(x,z), \quad &  \$ \deq  \omega - \omega_\0, \\[4pt]
 N \deq  N_\Omega^\omega(x,z), & D \deq   \DD_\Omega^\omega[\square](z).  
\end{matrix}
\end{equation}
The characteristic values of $\HH_0$ are the $\omega \in \CC$ for which there exist $\phi \in L^2(\Omega), \psi \in L^2(B)$, at least one of them non-zero, s.t.
\begin{equation*}
\begin{cases}
A\phi - \Gamma \int_{\pd B} \psi = 0,\\
D\phi + S \psi = 0.
\end{cases}
\end{equation*}
Applying $S^{-1}$ and integrating the second equation of the system we obtain
$$ \int_{\pd B} \psi = s_0 D \phi. $$
Substituting this back into the first equation, we have that the characteristic values of the system correspond to the characteristic values of the operator
$$H \deq A - s_0 \Gamma D.$$
Therefore the coefficient we are looking for will be given by
\begin{equation}
\label{proof:1stTermIntegralForm}
\tag{E}
\tr \oint_{\pd V} (\omega - \omega_\0)^l H^{-1}(\omega) H'(\omega) \, \d \omega,
\end{equation}
where $'$ denotes differentiation w.r.t. $\omega$.
A straightforward calculations shows that
\begin{align*}
H^{-1} = &~ (I - s_0 N D)^{-1} A^{-1} = \sum_{m=0}^{\infty} (s_0 ND)^m A^{-1},\\
H' = &~ A' - \Gamma' s_0 D - \Gamma (s_0 D)',
\end{align*}
then
\begin{align*}
H^{-1}H' = &~ A^{-1}A' + \sum_{m=1}^{\infty} (s_0 ND)^m A^{-1}A' - (s_0 ND)^{m-1} A^{-1} \Gamma' s_0 D - (s_0 ND)^{m-1}N (s_0 D)'.
\end{align*}
Since $A$ is analytic in $V$ and $A^{-1}$ has a simple pole at $\omega_\0$,
$$\oint_{\pd V} \$^l A^{-1}A' = 0.$$
Then 
\begin{align*}
\eqref{proof:1stTermIntegralForm} = &~ \sum_{m=1}^{\infty} \tr \oint_{\pd V} \$^l \left( (s_0 ND)^m A^{-1}A' - N^{m-1} (s_0 D)^m A^{-1} \Gamma' - N^m (s_0 D)^{m-1} (s_0 D)' \right)\\
= &~ \sum_{m=1}^{\infty} \oint_{\pd V} \$^l \left(s_0^m (DN)^{m-1} D A^{-1}A' A^{-1}\Gamma - s_0^m(DN)^{m-1} D A^{-1} \Gamma' - (s_0 DN)^{m-1} (s_0 D)'N \right).
\end{align*}
Since $(A^{-1})' = - A^{-1}A' A^{-1},$ by applying multiple times the chain rule we obtain
\begin{align*}
\eqref{proof:1stTermIntegralForm} =
&~ - \sum_{m=1}^{\infty} \dfrac{1}{m}   \oint_{\pd V} \$^l \left( (s_0 DN)^m \right)'.
\end{align*}
Then, by an integration by parts followed by a binomial expansion of $(DN)^m$, we have that
\begin{align}
\nonumber
\eqref{proof:1stTermIntegralForm} = &~  \sum_{m=1}^{\infty}\dfrac{l}{m}   \oint_{\pd V} \$^{l-1} (s_0 DN)^m \\
\label{proof:DNExpandedSum}
= &~  \sum_{m=1}^{\infty}\dfrac{l}{m}   \oint_{\pd V} \$^{l-1} s_0^m \Big( \dfrac{t_\omega}{\$} + r_\omega \Big)^m \\
\nonumber
= &~  \sum_{m=1}^{\infty}\dfrac{l}{m} \sum_{k=0}^{m} \binom{m}{k}   \oint_{\pd V} \dfrac{1}{\$^{k-l+1}} s_0^m t_\omega^k r_\omega^{m-k}.
\end{align}
Since the only pole in $V$ of the integrand is $\omega_\0$, by applying the residue theorem we can cancel each addend of the sum in $k$ except the one corresponding to a pole of order $1$, obtaining
\begin{align*}
\eqref{proof:1stTermIntegralForm} = &~  l t_{\omega_\0}^l r_{\omega_\0}^{-l} \sum_{m=l}^{\infty}\dfrac{1}{m} \binom{m}{l}  (s_0  r_{\omega_\0})^m .
\end{align*}
A final application of the identity
$$\sum_{m=l}^{\infty}\dfrac{1}{m} \binom{m}{l}  x^{m} = \dfrac{1}{l} \left(\dfrac{x}{1-x} \right)^{l},$$
leads to the formula in the thesis.

\end{proof}

\begin{remark}
\label{re:tildettilder}
The expansion for \eqref{eq:AsympExppl} still holds if we change both $t_{\omega_\0}, r_{\omega_\0}$ with respectively $\tilde t_{\omega_\0}, \tilde r_{\omega_\0}$ defined as
\begin{align*}
\tilde t_\omega \deq &~ \dfrac{U_\0(z)^2}{\omega+ \omega_\0},\\
\tilde r_\omega \deq &~ - U_\0(z) \cdot \int_\Omega \Gamma_\omega(z-y) U_\0(y)  \, \d y + \sum_{\substack{j=1 \\ j \neq \0}}^{\infty} \dfrac{U_j(z)  \cdot \DD_\Omega^{\omega}[U_j](z)}{\omega^2 - \omega_j^2}.
\end{align*}
This can be seen to hold true by regrouping the sum in \eqref{proof:DNExpandedSum} in the previous proof as
$$\DD_\Omega^\omega[N_\Omega^\omega(\cdot, z)](z) = \dfrac{\tilde t_\omega}{\omega - \omega_\0} + \tilde r_\omega.$$
\end{remark}

\subsection{First order term}
\begin{lemma}
\label{FirstOrderTerm}
The coefficient of the $\eps$ term in the expansion of $p_l$ is null.
\end{lemma}

\begin{proof}
With the notation introduced in \eqref{proof:1stTermRenames},
$$\HH_{1} \deq \sum_{|\a|=1} \begin{pmatrix}
0 &  (\pd^\a \Gamma) \int_{\pd B} y^\a \square  \\
(\pd^\a D) x^\a   & 0
\end{pmatrix}. $$
By applying the blockwise inversion formula 
$$ \begin{pmatrix}
W & X\\
Y & Z
\end{pmatrix}^{-1} = \begin{pmatrix}
(W-XZ^{-1}Y)^{-1} & -W^{-1} X(Z-YW^{-1}X)^{-1}\\
-Z^{-1}Y(W-XZ^{-1}Y)^{-1} & (Z-YW^{-1}X)^{-1}
\end{pmatrix}$$
to calculate $\HH_0^{-1}$, and rewriting the inverses of sums of operators in a Neumann series, we obtain
\begin{align*}
\HH_0^{-1} = &~
\begin{pmatrix}
A^{-1} & N \int_{\pd B}S^{-1}\\
-S^{-1}[1] DA^{-1} & -S^{-1}[1]\int_{\pd B}S^{-1} 
\end{pmatrix}\\
&~ + \sum_{m=0}^{\infty}
\begin{pmatrix}
(s_0 ND)^m A^{-1} & N (s_0 DN )^m \int_{\pd B} S^{-1}\\
-S^{-1}[1] (s_0 DN)^m DA^{-1} & -S^{-1}[1](s_0)^{m-1}(DN )^m \int_{\pd B} S^{-1}
\end{pmatrix} .
\end{align*}
A straightforward computation leads to
\begin{align*}
\HH_0^{-1} \HH_{1} = &~ \begin{pmatrix}
N s_\a (\pd^\a D)  & A^{-1} (\pd^\a \Gamma) \int_{\pd B} y^\a \square   \\
 -S^{-1}[1] s_\a  (\pd^\a D)  & -S^{-1}[1] DA^{-1} (\pd^\a \Gamma) \int_{\pd B} y^\a \square \, \d \sigma(y)
\end{pmatrix} \\
&~ +
\sum_{m=0}^{\infty} \sum_{|\a|=1}  \begin{pmatrix}
N  (s_0 DN)^m  s_\a (\pd^\a D) 
&
(s_0 ND)^m A^{-1} (\pd^\a \Gamma) \int_{\pd B} y^\a \square  \\
-S^{-1}[1]( DN)^m (s_0)^{m-1}s_\a (\pd^\a D) 
& 
-S^{-1}[1] (s_0 DN)^m DA^{-1} (\pd^\a \Gamma) \int_{\pd B} y^\a \square  \end{pmatrix}.
\end{align*}
Then, from the fact that $s_\a = 0$ for $|\a|=1$, we have that the coefficient of $\eps$ is
\begin{align*}
\tr \oint_{\pd V} \$^{l-1} \HH_0^{-1} \HH_{1} = s_\a (\dots)  = 0 .
\end{align*}
\end{proof}

\subsection{A proposal for an automated algorithm for higher order terms}
\label{sect:Higherorderproposal}

Let $a = DN$, $\phi = -S^{-1}[1]$.
We use $(m)^+$ to indicate the positive part of $m$, and the symbol $\sim$ to indicate that two operators have the same characteristic values.
Then we can rewrite
$$\HH_0^{-1} = \sum_{m=0}^{\infty}
\begin{pmatrix}
N s_0^m a^{(m-1)^+} D A^{-1} & N (s_0 a )^m\int_{\pd B} S^{-1}\\
\phi (s_0 a)^m DA^{-1} & \phi s_0 ^{(m-1)^+} a^m \int_{\pd B} S^{-1}
\end{pmatrix} ,
$$
$$\HH_{n} = \sum_{|\a|=n} \begin{pmatrix}
0 &  (-1)^{n+1} (\pd^\a \Gamma) \int_{\pd B} y^\a \square (y) \ \d \sigma(y) \\
(\pd^\a D) x^\a  & \XX_n
\end{pmatrix}.$$
An explicit computation leads to
\begin{multline*}
(-1)^{n+1} \HH_0^{-1} \HH_n  \\[6pt]
= \sum_{m=0}^{\infty} \sum_{|\a|=n}   \begin{pmatrix}
N (s_0 a )^m s_\a \pd^\a D
&
(-1)^{n+1} Ns_0 ^m a^{(m-1)^+} D  (\pd^\a N) \int_{\pd B} y^\a \square + N (s_0 a )^m\int_{\pd B} S^{-1} \XX_n \\
\phi s_0 ^{(m-1)^+} a^m \int_{\pd B} s_\a \pd^\a D
&
(-1)^{n+1} \phi (s_0 a)^m D (\pd^\a N) \int_{\pd B} y^\a \square +  \phi s_0 ^{(m-1)^+} a^m \int_{\pd B} S^{-1} \XX_n
\end{pmatrix}.
\end{multline*}
Since the upper and lower rows are respectively projections on the function $N$ and on the function $\phi$,
\begin{align*}
(-1&)^{n+1} \HH_0^{-1} \HH_n \\
& \sim \sum_{m=1}^{\infty} \sum_{|\a|=n}  \begin{pmatrix}
(s_0 a )^m s_\a \pd^\a DN 
&
(-1)^{n+1} s_0 ^m a^{(m-1)^+} D  (\pd^\a N) s_{0,\a} + (s_0 a )^m\int_{\pd B} S^{-1} \XX_n \phi \\
s_0 ^{(m-1)^+} a^m s_\a \pd^\a DN
&
(-1)^{n+1} (s_0 a)^m D (\pd^\a N) s_{0,\a}  +  s_0 ^{(m-1)^+} a^m \int_{\pd B} S^{-1} \XX_n \phi
\end{pmatrix}
\\
& = X \begin{pmatrix}
c_n 
&
(1 /a) d_n  +  q \\
(1/s_0  )c_n
&
d_n  +  (1/s_0 ) q
\end{pmatrix},
\end{align*}
where
$$X = \sum_{m=0}^\infty (s_0 a)^m = \dfrac{1}{1-s_0 a} =  \dfrac{\$ }{\$(1-s_0  r_\omega) -s_0 t_\omega}, $$ $$ c_n = \sum_{|\a|=n} s_\a \pd^\a DN, \quad d_n = (-1)^{n+1} \sum_{|\a|=n}  D \pd^\a N s_{0,\a}, \quad q = \int_{\pd B} S^{-1}\XX_n \phi.$$
Therefore
\begin{equation}
\label{eq:HighOrderTransMatrix}
(-1)^{n+1} \HH_0^{-1} \HH_{n_1} \dots \HH_0^{-1} \HH_{n_k} \sim X^k
\prod_{j=1}^k \begin{pmatrix}
c_{n_j} 
&
(1 /a) d_{n_j}  +  q \\
(1/s_0  )c_{n_j}
&
d_{n_j}  +  (1/s_0 ) q
\end{pmatrix}.
\end{equation}
Suppose that the coefficients of the matrix in \eqref{eq:HighOrderTransMatrix} can be rewritten explicitly in terms of sums of powers of the singularity $ \$ $ (this is a delicate part, as it is non trivial to explicit the singularity in $\pd^\a D N$ and $D \pd^\a N$ for general $\a$).
By selecting only the powers which sum up to $l-k$, we could then derive a constant $P$ s.t.
\begin{equation*}
\tr \oint_{\pd V} \$^{l-1} \HH_0^{-1} \HH_{n_1} \dots \HH_0^{-1} \HH_{n_k} = \dfrac{P}{( -s_0t_\omega)^k}.
\end{equation*}
Substituting this result back in the expression for $p_l$, we would thus have a method to compute any of the coefficients of the expansion in $\eps$.

\begin{remark}
\label{OnlyOddTerms}
If $k=|\a|$ is odd, we have $s_{\a,0}= s_{0,\a} = \XX_{|\a|} = 0 $, and thus the matrix in \eqref{eq:HighOrderTransMatrix} will have zeros on the diagonal and in the left lower corner.
Therefore for $k$ odd the coefficients of $\eps^k$ in the expansion of $p_l$ will always be zero.
\end{remark}

\section{Results for special cases}
\label{sect:Results}

We collect in this final section some interesting results which follow directly, or with minor algebraic manipulations, from Lemmas \ref{ZeroOrderTerm}, \ref{FirstOrderTerm} and Example \ref{ex:splittingp1p2}.

\paragraph{Simple eigenvalue}
Suppose $\omega_\0 $ is simple.
Then:
\begin{itemize}
\item we have
\begin{align*}
\omega_{\eps,1} - \omega_\0 = &~ \dfrac{t_{\omega_\0} }{ 1/s_0 - r_{\omega_\0}}  +  O(\eps^2),
\end{align*}
where, recalling Definition \ref{def:genCapacitytr},
\begin{align*}
1/s_0 =  - \dfrac{\ln(\eta_0 \omega_\0 \eps)}{2 \pi} - \ln \capac \pd B, \quad 
t_{\omega_\0} = &~ \dfrac{U_\0(z) \cdot \DD_\Omega^{\omega_\0}[U_\0](z)}{2 \omega_\0}, \quad
r_{\omega_\0} = \sum_{\substack{j=1 \\ j \neq \0}}^{\infty} \dfrac{U_j(z)  \cdot \DD_\Omega^{\omega_\0}[U_j](z)}{\omega_\0^2 - \omega_j^2}.
\end{align*}
\item
From Remark \ref{OnlyOddTerms} we know that there will be no terms $\eps^k$ with $k$ odd in the expansion.
\item
For $\eps$ small enough, we can deduce that $\omega_{\eps,1} \geq \omega_{\0} $.
\item
By considering an expansion in $1/ \ln(\eps)^n$, and substituting $\tilde t_\omega, \tilde r_\omega $ (as defined in Remark \ref{re:tildettilder}) to $t_\omega, r_\omega$, we obtain
$$\omega_{\eps,1} - \omega_\0 = - \dfrac{ 1}{\ln(\eps)} \dfrac{\pi U_\0(z)^2}{ \omega_\0 }  + O(\dfrac{1}{\ln(\eps)^{2}}),$$
which, once substituted $ 2 \omega_\0 \simeq \omega_{\epsilon,1} + \omega_{\0} $, is exactly \eqref{eq:oldomegaeps-omega0}.
\end{itemize}

\paragraph{Double eigenvalue}
If $\omega_\0 $ has double multiplicity then
\begin{align*}
\omega_{\eps,2} - \omega_{\0} = &~  \dfrac{t_{\omega_\0}}{1/s_0 - r_{\omega_\0}} + O(\eps^2),\\
\omega_{\eps,1} - \omega_{\0} = &~  O(\eps^2).
\end{align*}

We notice that if $z$ is on a nodal set of $U_\0$ (i.e. $u_{\0,1} \dots u_{\0,m_\0}$ are all zero at $z$) then $t_{\omega_\0} = 0$, and thus in both the cases of a simple or a double eigenvalue, the splitting order will be $O(\eps^2)$.

\paragraph{Disk domain and  disk inclusion}

Let $\Omega $ be the unit disk and let $\omega_\0^2$ be its first non-zero eigenvalue.
It is known that $\omega_\0^2$ is given by the first root of the derivative of the Bessel function $J_1$ and  has double multiplicity.
Suppose that also the rescaled inclusion $B$ is a unit disk.
First we compare results obtained through the multipole expansion method with the $\eps^2$ error theoretized by our formula for $\omega_{\eps,1} - \omega_{\0}$.

The multipole expansion is implemented by writing two polar coordinate systems, one centered in the center of $\Omega$ and one in $z$, and exploiting Graf's summation formula for Bessel function to rewrite the eigenvalue problem as a root finding problem for a complex valued function.

\begin{figure}[!htb]
\centering
\includegraphics[scale=.6]{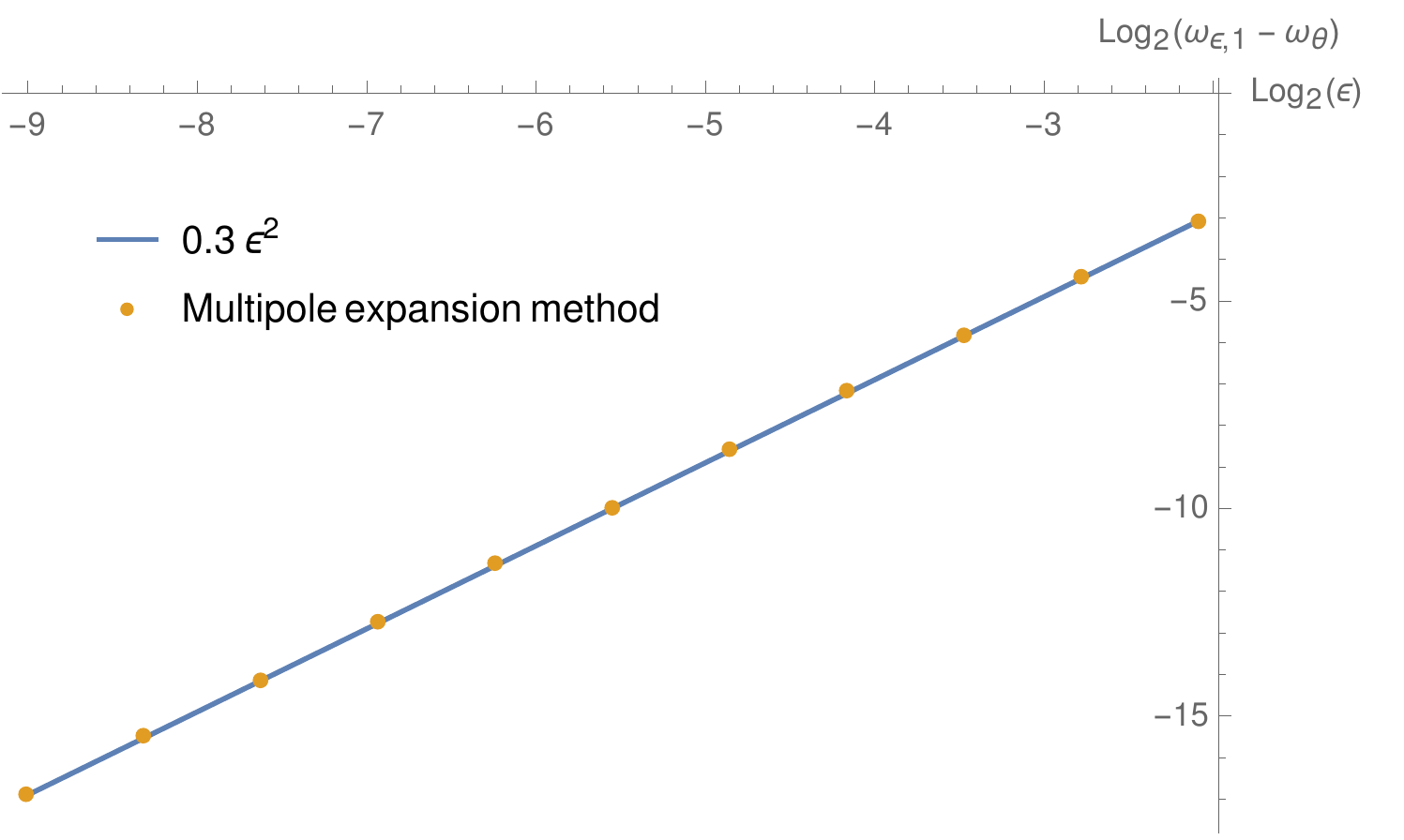}
\caption{The size of the inclusion varies, the center is fixed (at $z = (.5,0)$).}
\end{figure}

Now we compare our asymptotic formulae for $\omega_{\eps,2} - \omega_{\0}$ with results obtained with the multipole expansion method.
The asymptotic formula is implemented numerically by truncating at a finite value the series defining $r_{\omega_\0}$ in \eqref{eq:definet0r0}, and approximating the boundary layer integrals in $s_0, r_{\omega_\0}$ with adaptive quadrature methods.

\newpage

\begin{figure}[!htb]
\centering
\includegraphics[scale=.53]{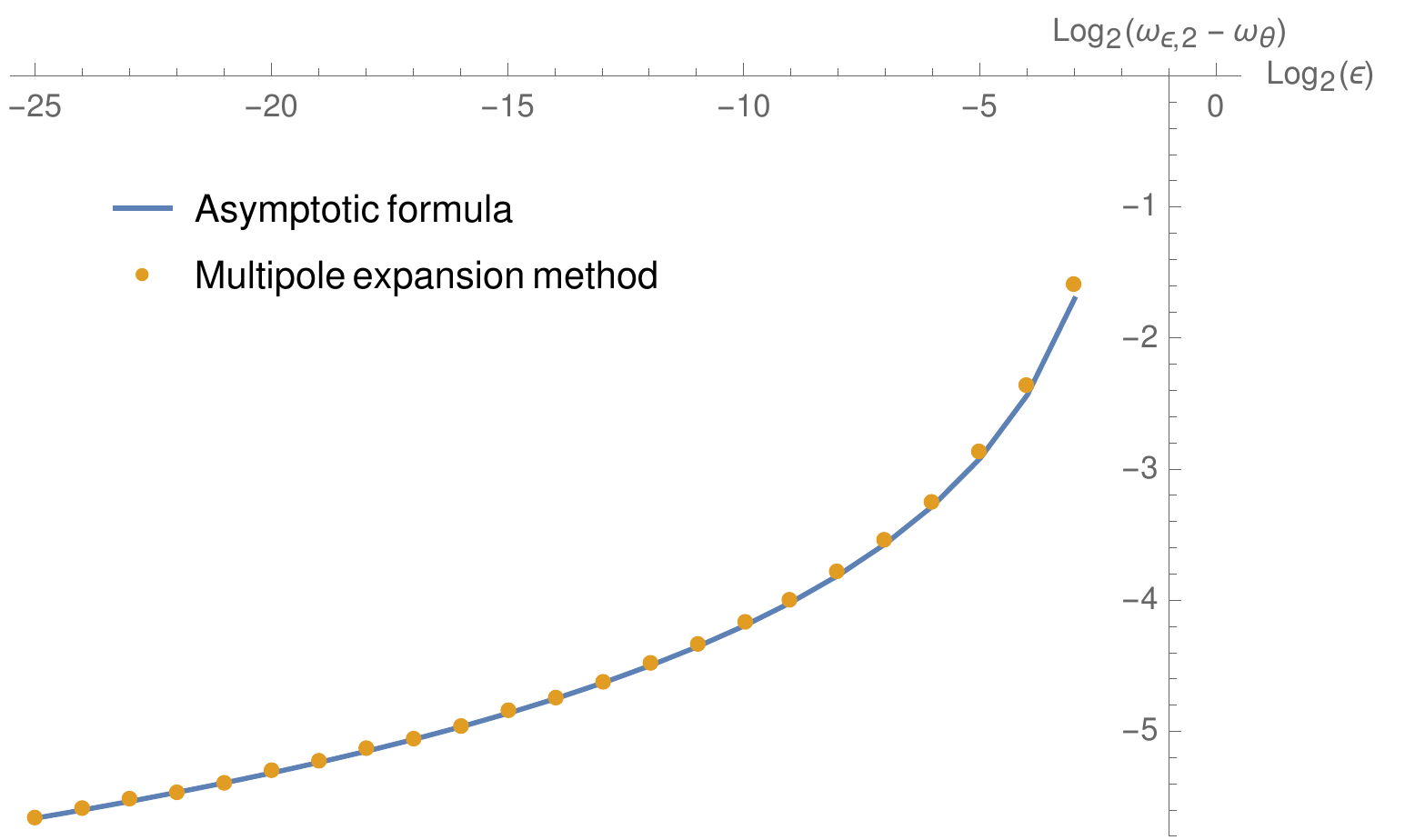}
\includegraphics[scale=.53]{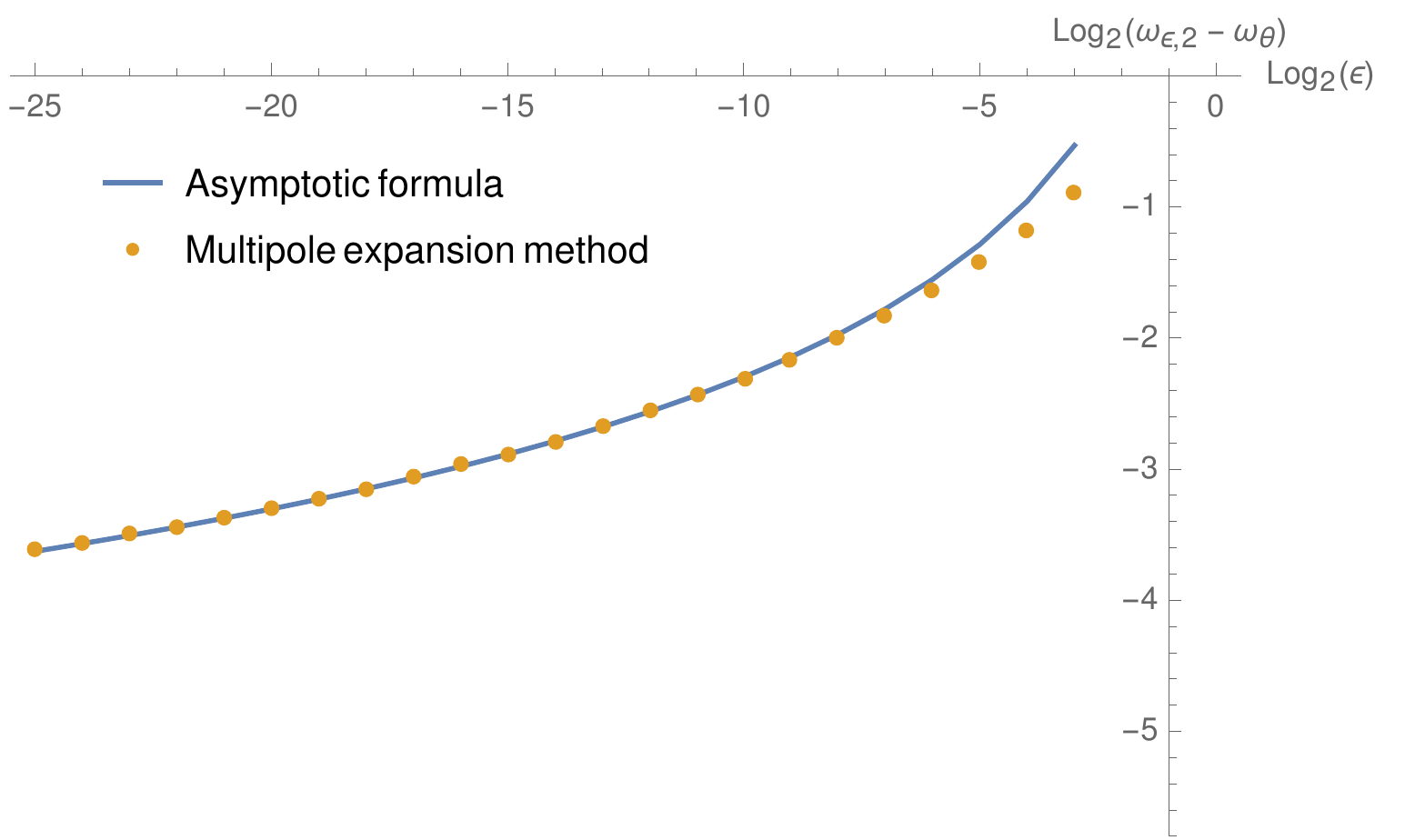}
\caption{The size of the inclusion varies, the center is fixed (left at $z = (.3,0)$, right at $z = (.8,0)$).}
\end{figure}
\begin{figure}[!htb]
\centering
\includegraphics[scale=.5]{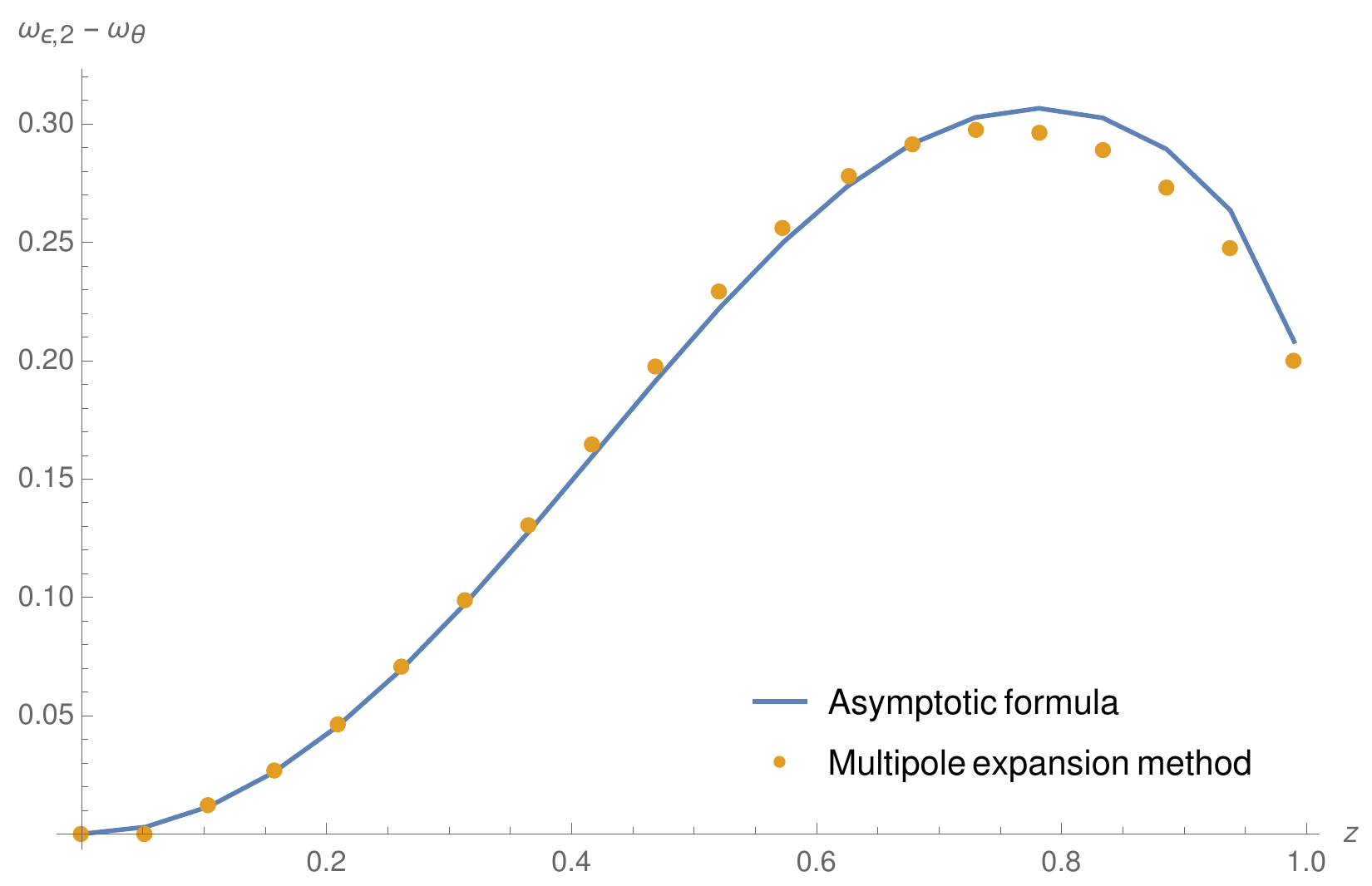}
\includegraphics[scale=.5]{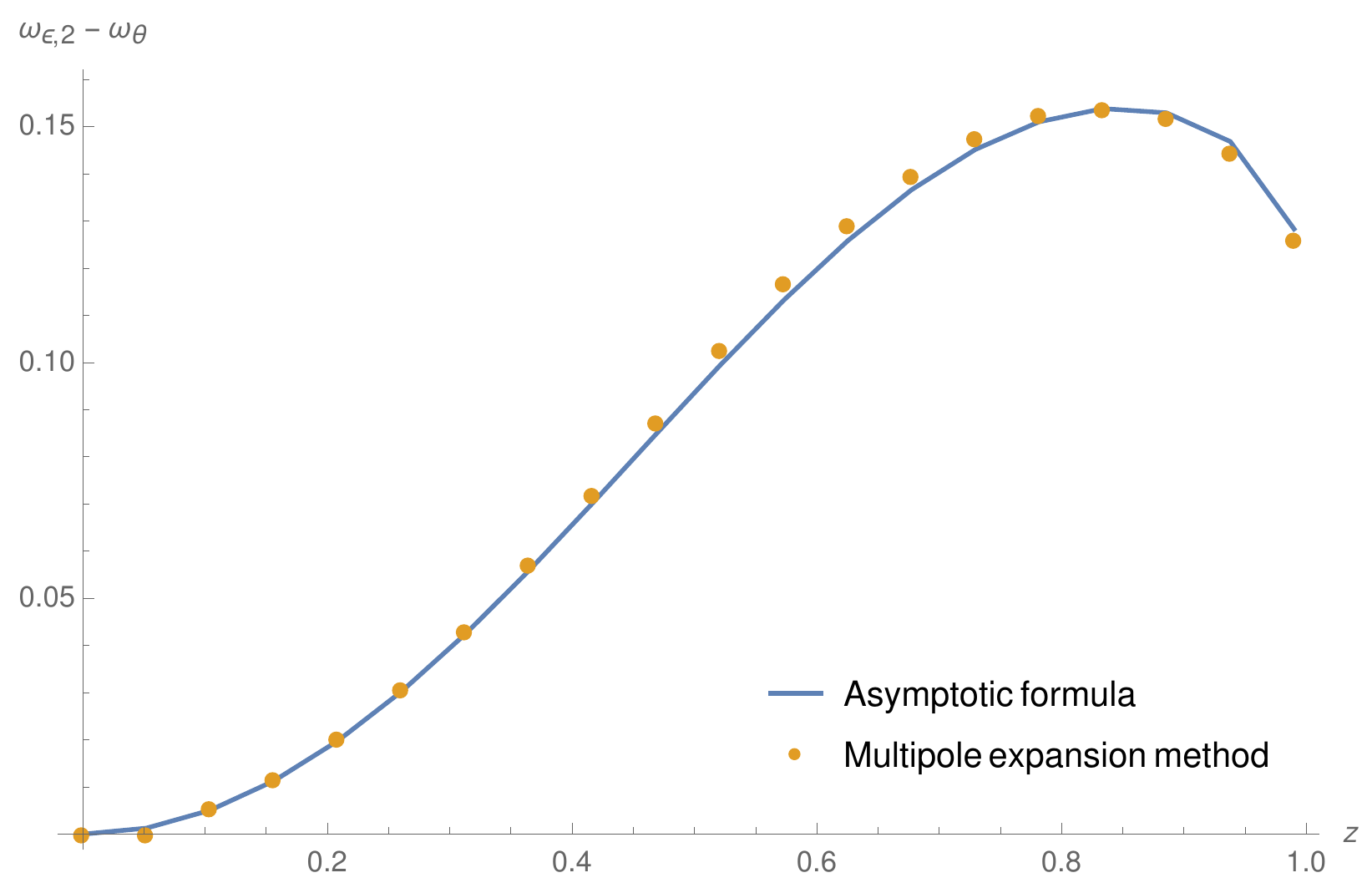}
\caption{The center of the inclusion varies, its size is fixed (left at $\eps = 10^{-2}$, right at $\eps = 10^{-4}$).}
\label{fig:digraph}
\end{figure}

We remark that the good resolution in $\eps$ and $z$ opens the possibility of inclusion reconstruction algorithms from the asymptotic formulae, which will be the topic of an upcoming paper.

\bibliography{ref}
\bibliographystyle{plain}
\end{document}